\documentclass[a4paper]{amsart}
\usepackage{amscd}\usepackage{amssymb}
\usepackage{color}
\usepackage[T1]{fontenc} 
\usepackage[utf8]{inputenc}
\usepackage{hyperref}
\usepackage[arrow,curve,matrix]{xy}
\usepackage{comment}
\usepackage{enumerate}
\usepackage{caption}

\usepackage[mathscr]{euscript}

\usepackage{mathpazo}
\usepackage{mathrsfs}
\usepackage{MnSymbol}

\usepackage{graphicx}
\usepackage{color}

\usepackage{tikz}
\usetikzlibrary{matrix,arrows,decorations.pathmorphing,decorations.markings,calc}
\usepackage{stackrel}

\usepackage{MnSymbol}
\usepackage{enumitem}

\hypersetup{
  colorlinks}

\allowdisplaybreaks

\DeclareFontFamily{OMS}{rsfs}{\skewchar\font'60}
\DeclareFontShape{OMS}{rsfs}{m}{n}{<-5>rsfs5 <5-7>rsfs7 <7->rsfs10 }{}
\DeclareSymbolFont{rsfs}{OMS}{rsfs}{m}{n}
\DeclareSymbolFontAlphabet{\scr}{rsfs}

\newcommand{\sF}{\scr{F}}

\newcommand{\sL}{\scr{L}}

\newcommand{\sO}{\scr{O}}

\newcommand{\bC}{\mathbb{C}}

\newcommand{\bP}{\mathbb{P}}

\newcommand{\sHom}{\scr{H}\negmedspace om}

\newcommand{\sm}{\text{sm}}

\newcommand{\red}{\text{red}}

\newcommand{\singu}{\text{sing}}

\theoremstyle{plain}    \newtheorem{thm}{Theorem}[section]

\numberwithin{equation}{thm}
\numberwithin{figure}{section}
\theoremstyle{plain}    

\newtheorem{lem}[thm]{Lemma}

\theoremstyle{plain}    
\newtheorem{prop}[thm]{Proposition}
\theoremstyle{remark}

\newtheorem{claim}[equation]{Claim} 
\newtheorem{notation}[thm]{Notation}

\newtheorem{add-ass}[equation]{Additional Assumption}

\definecolor{tomato}{RGB}{180,62,39}
\definecolor{forrest}{RGB}{81,133,49}
\definecolor{lighttomato}{RGB}{253,65,65}
\definecolor{lightforrest}{RGB}{145,237,87}
\definecolor{mygreen}{RGB}{40,104,69}
\definecolor{mygreen2}{RGB}{3,149,39}
\definecolor{darkolivegreen}{RGB}{102,118,75}
\definecolor{cranegreen}{RGB}{102,118,75}
\definecolor{mydarkblue}{RGB}{10,92,153}
\definecolor{myblue}{RGB}{57,222,186}
\definecolor{pinkish}{RGB}{213,83,222}
\definecolor{colD}{RGB}{213,83,222}
\definecolor{defb}{RGB}{213,83,222}
\definecolor{goldenrod}{RGB}{225,115,69}
\definecolor{mauve}{RGB}{224, 176, 255}
\definecolor{fuchsia}{RGB}{255, 0, 255}
\definecolor{lavender}{RGB}{230, 230, 250}
\definecolor{gold}{RGB}{255, 215, 0}
\definecolor{orange}{RGB}{255, 127, 0}
\definecolor{maroon}{RGB}{123, 17, 19}
\definecolor{brightmaroon}{RGB}{195, 33, 72}
\definecolor{richmaroon}{RGB}{176, 48, 96}
\definecolor{green}{RGB}{3,149,39}

\date{\today}

\author{Clemens J\"order}
\address{Clemens J\"order, Mathematisches Institut, Albert-Ludwigs-Universit\"at
  Freiburg, Eckerstraße 1, 79104 Freiburg im Breisgau, Germany}
\email{\href{mailto:c.joerder@web.de}{c.joerder@web.de}}

\thanks{The author gratefully acknowledges support by the DFG-Forschergruppe 790
  ``Classification of Algebraic Surfaces and Compact Complex Manifolds''.}


\title{A weak version of the Lipman-Zariski conjecture}

\begin{document}

\begin{abstract}
Let $X$ be a normal complex space such that the tangent sheaf $T_X$ is locally free and locally admits a basis consisting of pairwise commuting vector fields. Then $X$ is smooth.
\end{abstract}

\maketitle
\tableofcontents

\section{Introduction}
The Lipman-Zariski conjecture~\cite{Lip65} asserts that a complex variety with locally free tangent sheaf is necessarily smooth. In this paper we prove a weak version of the conjecture for complex spaces assuming an additional feature of the tangent sheaf of complex manifolds.

\begin{thm}[Weak version of the Lipman-Zariski conjecture]\label{thm-main}
Let $X$ be a normal complex space such that the tangent sheaf $T_X$ is locally free and locally admits a basis $v_1,\cdots,v_n$ consisting of pairwise commuting vector fields, i.e., $[v_i,v_j]=0$ for all $1\leq i,j\leq n$. Then $X$ is smooth.
\end{thm}

As in many special cases of the conjecture proved so far, our result relies on an extension theorem for differential forms. The precise statement is the following.

\begin{prop}[Extension of closed $1$-forms]\label{prop-extension}
Let $X$ be a normal complex space and let $\alpha\in\Gamma(X_\sm,\Omega^1_{X_\sm})$ be a closed differential form defined on the smooth locus $X_\sm\subset X$, i.e., $d\alpha=0$. Then $\alpha$ extends to any resolution of singularities of $\pi:\tilde{X}\to X$, i.e., there exists a section $\tilde{\alpha}\in H^0(\tilde{X},\Omega^1_{\tilde{X}})$ such that $\tilde{\alpha}|_{\pi^{-1}(X_\sm)}=\pi|_{\pi^{-1}(X_\sm)}^*(\alpha)$.
\end{prop}

 Proposition~\ref{prop-extension} as it stands does not hold for differential forms of higher degree. Counterexamples in degree $p\geq 2$ are given by Gorenstein non-canonical singularities of dimension $p$, e.g., a cone over a cubic curve in $\bP^2$.

Throughout the paper we make use of the following notation.
\begin{notation}
A resolution of singularities is a proper surjective holomorphic map $\pi:\tilde{X}\to X$ between a complex manifold $\tilde{X}$ and a reduced complex space $X$ such that there exists a nowhere dense analytic subset $A\subset X$ with nowhere dense preimage $\pi^{-1}(A)\subset\tilde{X}$ and $\pi^{-1}(X\backslash A)\to X\backslash A$ is an isomorphism. The holomorphic map $\pi$ is called a strong resolution if we can choose $A=X_\singu$ and $\pi^{-1}(A)_\red$ is a divisor with simple normal crossings.

We denote by $\Omega^i_X$ the usual sheaf of K\"ahler differential forms of degree $i\geq 0$ on a complex space $X$ as defined in~\cite[Def.~1]{Rei67}. If $X$ is normal and $j:X_\sm\to X$ is the inclusion of the smooth locus, we denote by $\Omega^{[i]}_X=j_*\Omega^i_{X_\sm}$ the sheaf of reflexive differential forms of degree $i$. Recall that the Second Riemann Removable Singularities Theorem~\cite[Thm.~71.12]{KK83} implies that the dual $\sF^*:=\sHom_{\sO_X}(\sF,\sO_X)$ of a coherent sheaf $\sF$ on $X$ satisfies $\sF^*=j_*(\sF|_{X_\sm}^*)$. The case $\sF=(\Omega^i_X)^*$ shows that $\Omega^{[i]}_X=j_*\Omega^i_{X_\sm}=j_*(\Omega^i_{X_\sm})^{**}=(\Omega^i_X)^{**}$ is a coherent sheaf.
\end{notation}

\subsection*{Previous results}
In the case of isolated singularities Theorem~\ref{thm-main} follows from \cite[Cor. 2]{OR88}. Observe that this includes in particular the two-dimensional case.

Previous results on the extension of differential forms are concerned with special kinds of singularities,~\cite[Sect. 2.3]{G80},~\cite{GKKP11,GKP12,Gr13}, and with differential forms of low degree in comparison to the codimension of the singular locus, see~\cite{SvS85,F88}. All these cases can be applied to the Lipman-Zariski conjecture, using the standard argument in~\cite[(1.6)]{SvS85}. Other approaches to the Lipman-Zariski conjecture can be found e.g. in~\cite{Hoch77,K11,Dr13}.

\subsection*{Acknowledgements}
The author was motivated to think about the Lipman-Zariski conjecture following interesting discussions with Patrick Graf, Daniel Greb and Sebastian Goette. The author would like to thank especially Stefan Kebekus, Daniel Greb and Patrick Graf for carefully reading a first version of this work. Karl Oeljeklaus kindly pointed to his and Richthofers results.

\section{Extension of closed differential forms of degree $1$}

Some parts of the following proof of Proposition~\ref{prop-extension} are inspired by the techniques in~\cite[§3]{F88}. However, the arguments in \textsl{loc. cit.} are formulated in the algebraic setting. Therefore we decided not to resort to these arguments during the proof.

We will use the following notation throughout the present section.

\begin{notation}\label{not-alpha-tilde}
Let $\pi:\tilde{X}\to X$ and $\alpha$ be as in Proposition~\ref{prop-extension}. We denote by $\tilde{\alpha}\in\Gamma(\pi^{-1}(X_\sm),\Omega^1_{\tilde{X}})$ the pull-back of $\alpha$ by $\pi|_{\pi^{-1}(X_\sm)}$.
\end{notation}

The following lemma is obvious in the algebraic setting. For the reader's convenience we include a short proof in the holomorphic case.

\begin{lem}\label{lem-r-i}
Let $\pi:\tilde{X}\to X$, $\alpha$ and $\tilde{\alpha}$ be as in Notation~\ref{not-alpha-tilde}. Let further $E_i$, $i\in I$, be the $\pi$-exceptional divisors. Then $\alpha$ has only poles along $E_i$, i.e., there exist minimal non-negative integers $r_i\geq 0$ such that $\tilde{\alpha}\in\Gamma(\tilde{X},\Omega^1_{\tilde{X}}(\sum_ir_i\cdot E_i))$.
\end{lem}

\begin{proof}
The sheaf $\Omega^1_{\tilde{X}}$ is a vector bundle so that $\tilde{\alpha}$ extends over analytic subsets of codimension $>1$ by Hartog's theorem. This already shows that $\tilde{\alpha}\in\Gamma(\tilde{X}\backslash\bigcup_iE_i,\Omega^1_{\tilde{X}})$.

Recall that by Grauert's theorem \cite[p.~235]{Grau60} the quotient $\Omega^{[1]}_X/\pi_*\Omega^1_{\tilde{X}}$ is a torsion coherent sheaf. In particular, at least locally on $X$, there exist a holomorphic function $f:X\to\bC$ such that $f$ is not identically zero on any irreducible component of $X$ and $f\cdot\alpha$ has zero image in $\Omega^{[1]}_{X}/\pi_*\Omega^1_{\tilde{X}}$. In other words,  $(f\circ\pi)\cdot\tilde{\alpha}\in\Gamma(\tilde{X},\Omega^1_{\tilde{X}})$ and this shows the claim.
\end{proof}

\begin{lem}\label{lem-discrete-case}
Let $\pi:\tilde{X}\to X$, $\alpha$ and $\tilde{\alpha}$ be as in Notation~\ref{not-alpha-tilde}. If $\tilde{\alpha}$ extends to $\pi^{-1}(X\backslash \{x\})\subset \tilde{X}$ for some $x\in X$, then it also extends to $\tilde{X}$.
\end{lem}

\begin{proof}
Since two resolutions are dominated by a third, the extendability of $\tilde{\alpha}$ does not depend on the particular choice of $\pi$. Furthermore extendability can be checked after shrinking $X$ to an arbitrarily small neighbourhood of $x$. By~\cite[Thm.~3.45]{Koll07} this implies that we may assume that $\pi$ is a projective resolution, i.e., there exists a closed analytic embedding $\tilde{X}\subset X\times \bP^N$ for some $N>0$. 

We prove the lemma by induction on $n=\dim_x(X)$.

\medskip
\textsl{Start of induction.} For $n\leq 2$, either $x\in X_\sm$ and the lemma is obvious, or $x\in X$ is a normal surface singularity. In the former case, we know by~\cite[Cor. (1.4)]{SvS85} and the closedness assumption that $\alpha$ extends to $\tilde{X}$.

\medskip
\textsl{Inductive step.} Suppose that $n\geq 3$. Let $E_1,\cdots,E_s$ be the $\pi$-exceptional divisors contained in the fiber $\pi^{-1}(\{x\})$. Using the notation of Lemma~\ref{lem-r-i}, we need to show that $r_1=\cdots=r_s=0$. Suppose to the contrary that this fails, say $r_1>0$. We show that this leads to a contradiction. Write $F:=E_1$ and $r:=r_1$.

Let $y\in F$ be a general point of $F$. Then there exists a smooth neighbourhood $F'\subset F$ of $y$ such that the differential form $\tilde{\alpha}$ induces a section of the vector bundle $\Omega^1_{\tilde{X}}(r\cdot F)|_{F'}$ that has no zero at $y\in F'$.

\begin{claim}\label{claim-cutting-down}
There exists a hyperplane $L\subset\bP^n$ such that $\tilde{H}:= X\times L\cap \tilde{X}$ and $F'_H:=F'\cap \tilde{H}$ satisfy the following:
\begin{enumerate}
 \item\label{it-y-contained} $y\in \tilde{H}$,
 \item\label{it-smoothness} the complex space $\tilde{H}$ is smooth in a neighbourhood of $\pi^{-1}(x)$, and 
 \item\label{it-poles} the pull-back $\tilde{\alpha}_H\in\Gamma(\tilde{H}\backslash\bigcup_iE_i,\Omega^1_{\tilde{H}})$ of $\tilde{\alpha}$ induces a section $\tilde{\alpha}_H|_{F'_H}\in\Gamma(F'_H,\Omega^1_{\tilde{H}}(r\cdot F'_H)|_{F'_H})$ that has no zero at $y\in F'_H$.
\end{enumerate}
\end{claim}

\begin{proof}[Proof of the Claim~\ref{claim-cutting-down}]
Let $\mathcal{H}$ be the set of all hyperplanes $L$ satisfying Item~(\ref{it-y-contained}) and let $\mathcal{H}_2,\mathcal{H}_3\subset \mathcal{H}$ be the subsets of hyperplanes satisfying Items~(\ref{it-smoothness}) and~(\ref{it-poles}), respectively. We need to show that $\mathcal{H}_2\cap \mathcal{H}_3\neq\emptyset$. It certainly suffices to prove that a general element $L\in \mathcal{H}$ is contained both in $\mathcal{H}_2$ and $\mathcal{H}_3$.

A general hyperplane $L\in \mathcal{H}$ is contained in $\mathcal{H}_2$ by Bertini's theorem, see the proof of~\cite[Cor. (II.7)]{Man82}.

To see the claim for $\mathcal{H}_3$, observe that there exists a non-empty open subset $\mathcal{H}_3'\subset\mathcal{H}$ such that $\dim_\bC T_yF'\cap T_y\tilde{H}=n-2>0$ and $T_y\tilde{H}\twoheadrightarrow N_{F'/\tilde{X}}|_y$ for any $L\in\mathcal{H}_3'$. Moreover, the resulting linear map $\bigoplus_{L\in\mathcal{H}_3'}T_y\tilde{H}\to T_y\tilde{X}$ is surjective. Dualizing and twisting shows that the linear map $\Omega^1_{\tilde{X}}(r\cdot F)|_y\to \prod_{L\in\mathcal{H}_3'}\Omega^1_{\tilde{H}}(r\cdot F'_H)|_y$ between vector spaces is injective. In particular, the non-zero vector $\tilde{\alpha}|_y\in \Omega^1_{\tilde{X}}(r\cdot F')|_y$ is mapped to a non-zero vector $\tilde{\alpha}_H|_y\in \Omega^1_{\tilde{H}}(r\cdot F'_H)|_y$ if $L$ lies in a suitable non-empty open subset $\mathcal{H}_3''\subset\mathcal{H}_3'$. This finishes the proof since $\mathcal{H}_3''\subset\mathcal{H}_3$.
\end{proof}

From now on, let $\tilde{H}$ and $\tilde{\alpha}_H$ be as in Claim~\ref{claim-cutting-down}. By Item~(\ref{it-smoothness}) of Claim~\ref{claim-cutting-down} we may shrink $X$ so that $\tilde{H}$ is smooth. Let $H\to \pi(\tilde{H})\subset X$ be the normalization. Recall that by~\cite[Prop. 71.15]{KK83} the resolution of singularities $\pi:\tilde{H}\to \pi(\tilde{H})$ factors through $H$. Then, the induced map $\pi_H:\tilde{H}\to H$ is a resolution of singularities of the normal complex space $H$. Write $x_H:=\pi_H(y)$. Then $\tilde{\alpha}_H$ is a closed differential form defined on the complement of $\pi_H^{-1}(x_H)$ in some open neighbourhood. There are two cases.
\begin{description}
 \item[Case 1: $x_H\in H_\sm$] Since $dim_{x_H}H_\sm=n-1\geq 2$, Hartog's theorem states that any differential form defined on a punctured neighbourhood of $x_H\in H$ extends across $x_H$. This applies to the form defined by $\tilde{\alpha}_H$, which contradicts Item~(\ref{it-poles}) of Claim~\ref{claim-cutting-down}, since $r>0$.
 \item[Case 2: $x_H\in H_\singu$] The inductive hypothesis applied to $\pi_H:\tilde{H}\to H$ shows that $\tilde{\alpha}_H$ extends to $\tilde{H}$, which yields again a contradiction to Item~(\ref{it-poles}) of Claim~\ref{claim-cutting-down}.
\end{description}
This finishes the proof of Lemma~\ref{lem-discrete-case}.
\end{proof}

\begin{proof}[Proof of Proposition~\ref{prop-extension}]
We maintain Notation~\ref{not-alpha-tilde} and the notation in Lemma~\ref{lem-r-i}. Using similar arguments as in the proof of Lemma~\ref{lem-discrete-case} we may assume that $\pi$ is a projective strong resolution.

Write $E=\bigcup_{r_i>0}E_i$ and $Z=\pi(E)\subset X$. We claim that $Z=\emptyset$. To this end, let us assume that $Z\neq\emptyset$ and show that this leads to a contradiction. By assumption, $r:=\text{max}\{r_i:i\in I\}>0$ is positive.

Observe that, in order to find a contradiction, we can shrink $X$ to an open subset that has non-empty intersection with $Z$. In so doing we may further assume that
\begin{enumerate}
 \item\label{it-irene} $Z\subset X$ is smooth and $\Omega^1_Z\cong\sO_Z\oplus\cdots\oplus\sO_Z$,
 \item the inclusion $Z\subset X$ admits a holomorphic left inverse $p:X\to Z$,
 \item\label{it-submersive} the map $p\circ \pi:\tilde{X}\to Z$ and its restrictions to $E_i$, $E_i\cap E_j$ are submersive for all $i,j$, and
 \item\label{it-isolated-problem} if we write $X_z:=p^{-1}(\{z\})$, $\tilde{X}_z:=\pi^{-1}(X_z)$ and $E_z:=E\cap\tilde{X}_z$ for $z\in Z$, then $\tilde{X}_z\to X_z$ is a strong resolution of a normal complex space and $E_z$ is an exceptional divisor mapped to $z\in X_z$. For normality, see the proof of~\cite[Thm. (II.5)]{Man82}.
\end{enumerate}
We will obtain the desired contradiction by considering for general $z\in Z$ the following commutative diagram
\begin{equation}\label{eqn-ex-seq-2}
 \begin{array}{l}\xymatrix{
0 \ar[r] &  (p\circ\pi)^*\Omega^1_Z(r\cdot E)\big|_{E_z} \ar[r] & \Omega^1_{\tilde{X}}(r\cdot E)\big|_{E_z}  \ar[r]\ar[d]_{{\tilde{\alpha}|_{E_z}\atop\downmapsto}\atop\tilde{\alpha}_z|_{E_z}} & \Omega^1_{\tilde{X}/Z}(r\cdot E)\big|_{E_z} \ar[r]\ar[d]^{\cong\text{ by Items~(\ref{it-submersive}),~(\ref{it-isolated-problem})}} & 0 \\
 & & \Omega^1_{\tilde{X}_z}(r\cdot E_z)\big|_{E_z}\ar[r]^-{\cong} & \Omega^1_{\tilde{X}_z/\{z\}}(r\cdot E_z)\big|_{E_z}, & 
} \end{array}
\end{equation}
which arises from the locally split exact sequence $0\to (p\circ\pi)^*\Omega^1_Z\to \Omega^1_{\tilde{X}}\to \Omega^1_{\tilde{X}/Z}\to 0$ of vector bundles by twisting and cutting down. By definition of $r$ and since $z$ is general, the section $\tilde{\alpha}\in\Gamma(\tilde{X},\Omega^1_{\tilde{X}}(r\cdot E))$ has non-zero restriction
\begin{equation}\label{eqn-antje}
0\neq \tilde{\alpha}|_{E_z}\in\Gamma\bigl(E_z,\Omega^1_{\tilde{X}}(r\cdot E)\big|_{E_z}\bigr).
\end{equation}
Observe that by Item~(\ref{it-isolated-problem}) restricting $\alpha$ yields a closed reflexive differential form $\alpha_z\in\Gamma(X_z,\Omega^{[1]}_{X_z})$ that extends to a differential form $\tilde{\alpha}_z\in\Gamma(\tilde{X}_z\backslash E_z,\Omega^1_{\tilde{X}_z})$. By Lemma~\ref{lem-discrete-case} we even have $\tilde{\alpha}_z\in \Gamma(\tilde{X}_z,\Omega^1_{\tilde{X}_z})\subset \Gamma(\tilde{X}_z,\Omega^1_{\tilde{X}_z}(r\cdot E_z))$ so that it induces the zero section
\begin{equation}\label{eqn-josef}
\tilde{\alpha}_z|_{E_z}=0\in\Gamma(E_z,\Omega^1_{\tilde{X}_z}(r\cdot E_z)|_{E_z}).
\end{equation}
Equations~(\ref{eqn-antje}) and~(\ref{eqn-josef}) together with Diagram~(\ref{eqn-ex-seq-2}) show that
\begin{equation}\label{eqn-marina}
\tilde{\alpha}|_{E_z}\in\Gamma\bigl(E_z,(p\circ\pi)^*\Omega^1_Z(r\cdot E)\big|_{E_z}\bigr) \subset \Gamma\bigl(E_z,\Omega^1_{\tilde{X}}(r\cdot E_z)\big|_{E_z}\bigr).
\end{equation}
Recall that by Item~(\ref{it-irene}) there exists an isomorphism
\begin{equation}\label{eqn-tabea}
{\textstyle \Gamma\bigl(E_z,(p\circ\pi)^*\Omega^1_Z(r\cdot E)\big|_{E_z}\bigr)\cong\bigoplus_{t=1}^{dim(Z)}\Gamma\bigl(E_z,\sO_{\tilde{X}_z}(r\cdot E_z)|_{E_z}\bigr).}
\end{equation}
Taking~(\ref{eqn-antje}),~(\ref{eqn-marina}) and~(\ref{eqn-tabea}) together we finally find the desired contradiction if the vector space on the right hand side of Equation~(\ref{eqn-tabea}) is shown to be zero. In the algebraic setting this follows directly from the negativity lemma in~\cite[Lem. 3.6.2(1)]{BCHM10}. In our analytic setting we can use the same proof as in \emph{loc. cit.}, replacing only the use of the Hodge index theorem on an algebraic surface by~\cite[p. 367]{Grau62}.
\end{proof}

\section{Proof of Theorem~\ref{thm-main}}
After shrinking $X$ if necessary, we may assume that $T_X=\sO_Xv_1\oplus\cdots\oplus\sO_Xv_n$ for pairwise commuting vector fields $v_i\in\Gamma(X,T_X)$, $1\leq i\leq n=dim(X)$. In other words, the Lie bracket $[v_i,v_j]=0$ vanishes for any $1\leq i,j\leq n$. Let $\alpha_i\in\Gamma(X,\Omega^{[1]}_X)$, $1\leq i\leq n$, be the dual basis.

In the following we denote the Lie derivative and the contraction along a vector field $v$ by $\sL_v$ and $\iota_v$, respectively. Given arbitrary indices $1\leq i,j,k\leq n$ we calculate
\[
0=\sL_{v_j}\delta_{i,k}=\sL_{v_j}\iota_{v_k}\alpha_i = \iota_{[v_j,v_k]}\alpha_i+\iota_{v_k}\sL_{v_j}\alpha_i=\iota_{v_k}\sL_{v_j}\alpha_i.
\]
Since $k$ is arbitrary we deduce that $\sL_{v_j}\alpha_i=0$. This in turn implies that
\[
0=\sL_{v_j}\alpha_i=d\iota_{v_j}\alpha_i + \iota_{v_j}d\alpha_i=d\delta_{i,j} + \iota_{v_j}d\alpha_i=\iota_{v_j}d\alpha_i.
\]
Since $j$ is arbitrary we obtain $d\alpha_i=0$. In particular, the differential form $\alpha_i$ extends to any resolution by Proposition~\ref{prop-extension}. Now we can argue as in~\cite[(1.6)]{SvS85}.

\bibliography{bibliography/general}{}

\def\cprime{$'$}
\providecommand{\bysame}{\leavevmode\hbox to3em{\hrulefill}\thinspace}
\providecommand{\MR}{\relax\ifhmode\unskip\space\fi MR}
\providecommand{\MRhref}[2]{%
  \href{http://www.ams.org/mathscinet-getitem?mr=#1}{#2}
}
\providecommand{\href}[2]{#2}
\begin{thebibliography}{BCHM10}

\bibitem[BCHM10]{BCHM10}
{\sc C.~Birkar, P.~Cascini, C.~D. Hacon, and J.~McKernan}: \emph{Existence of
  minimal models for varieties of log general type}, J. Amer. Math. Soc.
  \textbf{23} (2010), no.~2, 405--468. {\sf\scriptsize 2601039 (2011f:14023)}

\bibitem[{Dru}13]{Dr13}
{\sc S.~{Druel}}: \emph{{The Zariski-Lipman conjecture for log canonical
  spaces}}, \href{http://arxiv.org/abs/1301.5910}{arXiv: 1301.5910 (math.AG)}
  (2013).

\bibitem[Fle88]{F88}
{\sc H.~Flenner}: \emph{Extendability of differential forms on nonisolated
  singularities}, Invent. Math. \textbf{94} (1988), no.~2, 317--326.
  {\sf\scriptsize 958835 (89j:14001)}

\bibitem[{Gra}13]{Gr13}
{\sc P.~{Graf}}: \emph{{An optimal extension theorem for 1-forms and the
  Lipman-Zariski conjecture}}, \href{http://arxiv.org/abs/1301.7315}{arXiv:
  1301.7315 (math.AG)} (2013).

\bibitem[Gra60]{Grau60}
{\sc H.~Grauert}: \emph{Ein {T}heorem der analytischen {G}arbentheorie und die
  {M}odulr\"aume komplexer {S}trukturen}, Inst. Hautes \'Etudes Sci. Publ.
  Math. (1960), no.~5, 64. {\sf\scriptsize 0121814 (22 \#12544)}

\bibitem[Gra62]{Grau62}
{\sc H.~Grauert}: \emph{\"{U}ber {M}odifikationen und exzeptionelle analytische
  {M}engen}, Math. Ann. \textbf{146} (1962), 331--368. {\sf\scriptsize 0137127
  (25 \#583)}

\bibitem[GKP13]{GKP12}
{\sc D.~{Greb}, S.~{Kebekus}, and T.~{Peternell}}: \emph{{Reflexive
  differential forms on singular spaces -- Geometry and Cohomology}}, Journal
  f\"ur die Reine und Angewandte Mathematik (Crelle's Journal), published
  electronically (2013).

\bibitem[GKKP11]{GKKP11}
{\sc D.~Greb, S.~Kebekus, S.~J. Kov{\'a}cs, and T.~Peternell}:
  \emph{Differential forms on log canonical spaces}, Publ. Math. Inst. Hautes
  \'Etudes Sci. (2011), no.~114, 87--169. {\sf\scriptsize 2854859}

\bibitem[Gre80]{G80}
{\sc G.-M. Greuel}: \emph{Dualit\"at in der lokalen {K}ohomologie isolierter
  {S}ingularit\"aten}, Math. Ann. \textbf{250} (1980), no.~2, 157--173.
  {\sf\scriptsize 582515 (82e:32009)}

\bibitem[Hoc77]{Hoch77}
{\sc M.~Hochster}: \emph{The {Z}ariski-{L}ipman conjecture in the graded case},
  J. Algebra \textbf{47} (1977), no.~2, 411--424. {\sf\scriptsize 0469917 (57
  \#9697)}

\bibitem[K{\"a}l11]{K11}
{\sc R.~K{\"a}llstr{\"o}m}: \emph{The {Z}ariski-{L}ipman conjecture for
  complete intersections}, J. Algebra \textbf{337} (2011), 169--180.
  {\sf\scriptsize 2796069 (2012d:14001)}

\bibitem[KK83]{KK83}
{\sc L.~Kaup and B.~Kaup}: \emph{Holomorphic functions of several variables},
  de Gruyter Studies in Mathematics, vol.~3, Walter de Gruyter \& Co., Berlin,
  1983, An introduction to the fundamental theory, With the assistance of
  Gottfried Barthel, Translated from the German by Michael Bridgland.
  {\sf\scriptsize 716497 (85k:32001)}

\bibitem[Kol07]{Koll07}
{\sc J.~Koll{\'a}r}: \emph{Lectures on resolution of singularities}, Annals of
  Mathematics Studies, vol. 166, Princeton University Press, Princeton, NJ,
  2007. {\sf\scriptsize 2289519 (2008f:14026)}

\bibitem[Lip65]{Lip65}
{\sc J.~Lipman}: \emph{Free derivation modules on algebraic varieties}, Amer.
  J. Math. \textbf{87} (1965), 874--898. {\sf\scriptsize 0186672 (32 \#4130)}

\bibitem[Man82]{Man82}
{\sc M.~Manaresi}: \emph{Sard and {B}ertini type theorems for complex spaces},
  Ann. Mat. Pura Appl. (4) \textbf{131} (1982), 265--279. {\sf\scriptsize
  681567 (85d:32020)}

\bibitem[OR88]{OR88}
{\sc K.~Oeljeklaus and W.~Richthofer}: \emph{Linearization of holomorphic
  vector fields and a characterization of cone singularities}, Abh. Math. Sem.
  Univ. Hamburg \textbf{58} (1988), 63--87. {\sf\scriptsize 1027433
  (91f:32042)}

\bibitem[Rei67]{Rei67}
{\sc H.-J. Reiffen}: \emph{Das {L}emma von {P}oincar\'e f\"ur holomorphe
  {D}ifferential-formen auf komplexen {R}\"aumen}, Math. Z. \textbf{101}
  (1967), 269--284. {\sf\scriptsize 0223599 (36 \#6647)}

\bibitem[SvS85]{SvS85}
{\sc J.~Steenbrink and D.~van Straten}: \emph{Extendability of holomorphic
  differential forms near isolated hypersurface singularities}, Abh. Math. Sem.
  Univ. Hamburg \textbf{55} (1985), 97--110. {\sf\scriptsize 831521
  (87j:32025)}

\end{thebibliography}
\bibliographystyle{bibliography/skalpha}

\end{document}